\documentclass[11pt]{article} % use larger type; default would be 10pt
\usepackage[utf8]{inputenc} % set input encoding (not needed with XeLaTeX)

\usepackage[letterpaper,margin=1in]{geometry} % to change the page dimensions

\usepackage{enumerate}
\usepackage{amsthm,amsmath,amsfonts,amssymb}

\usepackage[dvipsnames]{xcolor}
\usepackage[colorlinks=true, pdfstartview=FitV, linkcolor=Blue, citecolor=Blue, urlcolor=Blue]{hyperref}

\usepackage[nottoc,notlot,notlof]{tocbibind}

\numberwithin{equation}{section}

\newtheorem{theorem}{Theorem}[section]
\newtheorem{conjecture}[theorem]{Conjecture}

\newtheorem{corollary}[theorem]{Corollary}
\newtheorem{proposition}[theorem]{Proposition}
\newtheorem{lemma}[theorem]{Lemma}

\theoremstyle{definition}
\newtheorem{definition}[theorem]{Definition}
\newtheorem{remark}[theorem]{Remark}

\renewcommand{\tilde}{\widetilde}
\renewcommand{\hat}{\widehat}

\usepackage{dsfont} % for indicator \indc
    \newcommand{\indc}{\mathds{1}}
    \renewcommand{\mathbb}{\mathds}

\newcommand{\N}{\ensuremath{\mathbb{N}}}

\newcommand{\R}{\ensuremath{\mathbb{R}}}
\newcommand{\T}{\mathbb{T}}
\newcommand{\e}{\mathrm{e}}

	\makeatletter
	\providecommand*{\diff}%
	{\@ifnextchar^{\DIfF}{\DIfF^{}}}
	\def\DIfF^#1{%
	\mathop{\mathrm{\mathstrut d}}%
	\nolimits^{#1}\gobblespace}
	\def\gobblespace{%
	\futurelet\diffarg\opspace}
	\def\opspace{%
	\let\DiffSpace\!%
	\ifx\diffarg(%
	\let\DiffSpace\relax
	\else
	\ifx\diffarg%
	\let\DiffSpace\relax
	\else
	\ifx\diffarg\{%
	\let\DiffSpace\relax
	\fi\fi\fi\DiffSpace}
    \newcommand{\dd}{\diff}

\let\originalleft\left
\let\originalright\right
\renewcommand{\left}{\mathopen{}\mathclose\bgroup\originalleft}
\renewcommand{\right}{\aftergroup\egroup\originalright}

\newcommand{\Var}{\mathrm{Var}}

\newcommand{\solP}{\mathcal{P}}

\title{Divergence-free drifts decrease concentration}

\author{Elias Hess-Childs\thanks{Department of Mathematical Sciences, Carnegie Mellon University.
{\footnotesize \href{mailto:aa@cims.nyu.edu}{ehesschi@andrew.cmu.edu}.}
}
\and
Renaud Raqu\'epas\thanks{Courant Institute of Mathematical Sciences,  New York University.
{\footnotesize \href{mailto:rr4374@nyu.edu}{rr4374@nyu.edu}.}
}
\and
Keefer Rowan\thanks{Courant Institute of Mathematical Sciences,  New York University.
{\footnotesize \href{mailto:keefer.rowan@cims.nyu.edu}{keefer.rowan@cims.nyu.edu}.}
}
}
\date{\today}

\begin{document}

\maketitle
\begin{abstract}
    We show that bounded divergence-free vector fields $u : [0,\infty) \times \R^d \to\R^d$ decrease the ``concentration''---quantified by the modulus of absolute continuity with respect to the Lebesgue measure---of solutions to the associated advection-diffusion equation when compared to solutions to the heat equation. In particular, for symmetric decreasing initial data, the solution to the advection-diffusion equation has (without a prefactor constant) larger variance, larger entropy, and smaller $L^p$ norms for all $p \in [1,\infty]$ than the solution to the heat equation. We also note that the same is \textit{not true} on~$\T^d$.


\end{abstract}

\section{Introduction}

In this paper, we study the concentration behavior of advection-diffusion equations with divergence-free drifts compared to that of the equation describing pure diffusion, i.e., the heat equation. To be more precise, for bounded divergence-free vector fields $u : [0,\infty)\times \R^d \to \R^d$, we compare the solutions $\theta : [0,\infty)\times \R^d \to[0,\infty)$ and $\varphi : [0,\infty) \times \R^d \to [0,\infty)$ to
\begin{align*}
    \begin{cases}
        \partial_t \theta - \Delta \theta + u \cdot \nabla \theta = 0\\
        \theta(0,\,\cdot\,) = \theta_0(\,\cdot\,),
    \end{cases}
    \quad \text{and} \qquad
    \begin{cases}
        \partial_t \varphi - \Delta \varphi = 0\\
        \varphi(0,\,\cdot\,) = \varphi_0(\,\cdot\,),
    \end{cases}
\end{align*}
respectively.
One particularly interesting limiting case we cover is when $\theta_0 = \varphi_0 =\delta_y$ for some $y \in \R^d$. In this case, we show that $\theta(t,\,\cdot\,)$ is always \textit{less concentrated than} $\varphi(t,\,\cdot\,)$, that is
\[\int_E \theta(t,x)\dd x \leq \int_{B_r(y)} \varphi(t,x)\dd x \quad \text{whenever} \quad |E| = |B_r(y)|\]
and $E$ is Borel. As a consequence, we conclude that for all $p\geq 1,$ $\|\theta(t,\,\cdot\,)\|_{L^p} \leq \|\varphi(t,\,\cdot\,)\|_{L^p}$ and that for the solution to the It\^o diffusion
\begin{equation*}
\begin{cases}
    \dd X_t = u(t,X_t)\dd t +\sqrt{2} \dd W_t\\
    X_0 =y,
\end{cases}\end{equation*}
we have that
\[\Var(X_t) \geq \Var(y+\sqrt{2} W_t).\]

\subsection{Statement of main results}

Denote the space of positive, finite Borel measures on $\mathbb{R}^d$ by $\mathcal{M}_+(\mathbb{R}^d)$. We also let $L^1_+(\R^d)$ denote the subset of $L^1(\R^d)$ given by nonnegative functions. For the $d$-dimensional Lebesgue measure (or volume) and a Borel set~$E\subseteq\R^d$, we often write $|E|$ for $\mathrm{vol}(E)$. We first define the central object of study, $C_\mu$, the \textit{modulus of absolute continuity of $\mu$ with respect to the Lebesgue measure $\mathrm{vol}$}. 

\begin{definition}
    For $\mu\in\mathcal{M}_+(\R^d)$ and $\alpha \geq 0$, we define
    \begin{equation}
    \label{eq:C-mu-def}
    C_\mu(\alpha):=\sup\{\mu(E): E \subseteq \R^d \text{ Borel, }|E|=\alpha\}.
    \end{equation}
    For $f \in L^1_+(\R^d)$, we abuse notation by denoting
    \[C_f(\alpha) := C_\mu(\alpha)\quad \text{where } \mu(\dd x) = f(x)\dd x.\]
\end{definition}

\begin{remark}\label{rem:C_mu-decomposition}
    If $\mu(\dd x)=f(x)\dd x+\beta(\dd x)$ is the Lebesgue decomposition of $\mu$ with respect to $\mathrm{vol}$, 
    then
    $C_\mu(\alpha)=C_f(\alpha)+\beta(\R^d)$.
\end{remark}

\begin{definition}
 For $\mu,\nu\in\mathcal{M}_+(\R^d)$, we say that \textit{$\mu$ is less concentrated than $\nu$}, denoted $\mu \preceq \nu$, if
    \[\forall \alpha \in [0,\infty),\quad C_\mu(\alpha) \leq C_\nu(\alpha).\]
We abuse notation for densities, writing $f \preceq g$ if the analogous relation holds for their associated measures.
\end{definition}

\begin{definition}
\label{def:symmetric-decreasing}
    We say that a function $f : \R^d \to \R$ is \emph{symmetric decreasing} if 
    \[f(x) = f(y) \text{ for all } |x| = |y| \quad\text{and}\quad f(x) \geq f(y) \text{ for all } |x| \leq |y|.\]
    We say that $\mu\in\mathcal{M}_+(\R^d)$ is \emph{symmetric decreasing} if
    \[\mu(\dd x) = f(x)\dd x + c \delta_0(\dd x)\]
    for some symmetric decreasing density $f$ and some $c\geq 0$.
\end{definition}

\begin{remark}
    The definition of symmetric decreasing measures may appear \textit{ad hoc}, but it arises naturally as the weak closure of measures with symmetric decreasing densities in the space of positive measures.
\end{remark}

We note that the above definitions are precisely given so that if $\mu \preceq \nu$ and $\nu$ is symmetric decreasing, then 
\[\mu(E) \leq \nu(B_r)\quad \text{whenever} \quad |E| = |B_r|\]
and $E$ is Borel. We now introduce notation for the solution operators to the heat equation and the advection-diffusion equation. 

\begin{definition}
    For $\mu\in\mathcal{M}_+(\R^d)$, let $\e^{t\Delta}\mu$ denote the time $t$ solution to the heat equation with initial condition $\mu$, that is 
    \[\partial_t \e^{t\Delta} \mu = \Delta \e^{t\Delta}\mu \quad \text{and} \quad \e^{0\Delta}\mu = \mu.\]
    For a divergence-free vector field $u : [0,\infty) \times \R^d \to \R^d$ with $u \in L^\infty_{t,x}$, let $\solP^u_t \mu$ denote the time $t$ solution to the advection-diffusion equation, that is
    \[\partial_t \solP^u_t \mu = \Delta \solP^u_t \mu  - u \cdot \nabla \solP^u_t \mu \quad\text{and}\quad \solP^u_0 \mu = \mu.\]
    We emphasize that we \textbf{do not} assume that $u$ is time-independent.
\end{definition}

We now state our main result, which makes precise the idea that advection-diffusion equations with divergence-free drift always decrease the ``concentration'' of the solution when compared to the heat equation. Note that of particular interest is the case that $\mu = \delta_y, \nu = \delta_0$, which compares the fundamental solution of the advection-diffusion equation to the heat kernel. 

\begin{theorem}
\label{thm:main-result}
    Let $\mu, \nu \in \mathcal{M}_+(\R^d)$ and let $u : [0,\infty) \times \R^d\to \R^d$ be a bounded, measurable, divergence-free vector field. If $\nu$ is symmetric decreasing and $\mu \preceq \nu$, then
    \[\solP_t^u \mu \preceq \e^{t\Delta} \nu\]
    for all $t\geq 0$.
\end{theorem}

\begin{remark}
If $\nu$ is not symmetric decreasing, the above relation may fail. See the discussion in Subsection~\ref{subseq:motivation} for an example.
\end{remark}

\begin{remark}
    One may wonder if for suitable vector fields $u$ the above relation is strict: $\solP_t^u \mu \preceq \e^{t\Delta}\nu$ and $\e^{t\Delta} \nu \not \preceq \solP^u_t \mu$. This is however not straightforward to deduce from our argument, as is further discussed in Subsection~\ref{ss:overview}.
\end{remark}

Due to the general relations proved in Section~\ref{sec:proof-of-cor}, we deduce the following bounds on $L^p$ norms, variances, and entropies.  The latter two follow the conventions given in Definition~\ref{def:variance-entropy-def}.

\begin{corollary}
\label{cor:main-cor}
    Let $u$ be a bounded, measurable, divergence-free vector field and let $\mu,\nu$ be probability measures. If $\nu$ is symmetric decreasing and $\mu \preceq \nu$,  then 
    \begin{equation}
    \label{eq:Lp-var-ineqs}
    \|\solP^u_t \mu\|_{L^p(\R^d)} \leq \|\e^{t\Delta} \nu\|_{L^p(\R^d)} \quad \text{and} \quad \Var(\solP_t^u\mu) \geq \Var(\e^{t\Delta} \nu)
    \end{equation}
    for all $p \in [1,\infty]$ and $t \geq 0$.
    If, in addition, $\int |x|^\beta\dd \mu(x) <\infty$ for some $\beta>0$, then 
    \begin{equation}
        \label{eq:entropy-ineq}
    H(\solP^u_t \mu) \geq H(\e^{t\Delta} \nu)
    \end{equation}
    for all $t > 0$.
\end{corollary}

The following corollary is a direct consequence of Corollary~\ref{cor:main-cor} applied with $p=2$ and the energy balance for the advection-diffusion equation:
\[\frac{\dd}{\dd t} \|\solP^u_t \mu\|_{L^2(\R^d)}^2 = -2 \|\nabla \solP^u_t \mu\|_{L^2(\R^d)}^2.\]
\begin{corollary}
\label{cor:second} 
    Suppose $\mu\in\mathcal{M}_+(\R^d)$ is symmetric decreasing and $u$ is a bounded, measurable, divergence-free vector field. Then for all $t\geq 0$
    \[\int_0^t\|\nabla\solP^u_s\mu\|_{L^2(\R^d)}^2\dd s\geq \int_0^t\|\nabla \e^{s\Delta}\mu\|_{L^2(\R^d)}^2\dd s.\]
\end{corollary}

We now note that as a direct consequence of Corollary~\ref{cor:main-cor}, using the linearity and positivity preserving properties of the solution operator $\solP^u_t$, we can also comment on initial data that are not necessarily positive.
\begin{corollary}
\label{cor:signed}
    Let $\mu \in \mathcal{M}(\R^d)$, the space of finite signed measures, and let $\mu = \mu_+ - \mu_-$ be its Hahn--Jordan decomposition. Then suppose $\nu_+, \nu_- \in \mathcal{M}_+(\R^d)$ are symmetric decreasing with $\mu_+ \preceq \nu_+$ and $\mu_- \preceq \nu_-$. Let $u$ be a bounded, measurable, divergence-free vector field. Then for all $p \in [1,\infty], t \geq 0$,
    \[\|\solP^u_t \mu\|_{L^p(\R^d)} \leq \Big(\|\e^{t\Delta} \nu_+\|_{L^p(\R^d)}^p + \|\e^{t\Delta} \nu_-\|_{L^p(\R^d)}^p\Big)^{\frac 1p},\]
    where the right-hand side is interpreted as a maximum for $p =\infty$.
\end{corollary}

It is also a direct consequence of Corollary~\ref{cor:main-cor} that we find the following inequality for solutions to the two-dimensional Navier--Stokes equation with vorticity of fixed sign. We note that initial conditions are chosen so that, by Sobolev embedding, $\nabla^\perp (-\Delta)^{-1} \omega \in L^\infty_{t,x}$ so that we can apply the above results.

\begin{corollary}
\label{cor:third}
    If $\omega$ is a solution to the vorticity form of the two-dimensional Navier--Stokes equation
    \[
    \begin{cases}
    \partial_t\omega-\Delta\omega+\nabla^{\perp}(-\Delta)^{-1}\omega\cdot\nabla\omega=0\\
    \omega(0)=\omega_0
    \end{cases}
    \]
    where $\omega_0 \in L^q(\R^2) \cap L^1_+(\R^2)$ for some $q>2$ is symmetric decreasing, then for all $t>0$ and $p\in[1,\infty]$
    \[\|\omega(t,\,\cdot\,)\|_{L^p(\R^2)}\leq \|\e^{t\Delta}\omega_0\|_{L^p(\R^2)}\quad \text{and}\quad \int_0^t \|\nabla \omega(s,\,\cdot\,)\|_{L^2(\R^2)}^2\dd s \geq \int_0^t \|\nabla \e^{s\Delta} \omega_0\|^2_{L^2(\R^2)}\dd s.\]
\end{corollary}

Our final main result is that the above inequalities fail on the torus $\T^d$. That is, there exists a divergence-free vector field $u$ which (in some sense) increases the concentration of the solution to the advection-diffusion equation compared to the solution to the heat equation.

\begin{theorem}
\label{thm:torus}
    There exists a divergence-free vector field $u : [0,\infty)\times \T^d \to \R^d$ such that $u \in C^\infty_{t,x}$ with compact support in time and for some $T>0$ and all $t \geq T$, we have that
    \begin{equation}
    \label{eq:L2-torus-inversion}
    \|\solP^u_t \delta_0\|_{L^2(\T^d)} > \|\e^{t\Delta} \delta_0\|_{L^2(\T^d)},
    \end{equation}
    where $\delta_0 \in \mathcal{M}_+(\T^d)$ is the delta mass at $0$. In particular, for all $t > T$, we have
    \begin{equation}
    \label{eq:torus-not-preceq}
    \solP^u_t \delta_0\not \preceq \e^{t\Delta} \delta_0.
    \end{equation}
\end{theorem}

\paragraph*{Acknowledgments}
We thank Almut Burchard and Gautam Iyer for their helpful suggestions. 
EHC was partially supported by NSF grant DMS-2342349. RR was partially funded by the Fonds de recherche du Qu\'ebec (section Nature et technologies). KR was partially supported by the NSF Collaborative Research Grant DMS-2307681 and the Simons Foundation through the Simons Investigators program.

\section{Discussion and background}

As is discussed in Subsection~\ref{ss:related}, similar objects---as well as similar arguments---have been used to study similar problems previously. However, the specific setting of drift-diffusion equations with divergence-free drifts---of substantial importance to the study of passive scalar turbulence---has so far not seen much attention. Despite the simplicity of the argument for an expert in rearrangements, these techniques are less familiar to the mathematical fluids community. As such, we provide a thorough exposition of the relevant tools for this specific problem, despite many of the lemmas being elementary and having previously appeared in various forms throughout the literature on rearrangements.

\subsection{Motivation and interpretation of results}\label{subseq:motivation}

Consider the It\^o diffusion on $\R^d$ with drift $u :[0,\infty) \times \R^d\to \R^d$ with $u \in L^\infty_{t,x},$
\begin{equation}
\label{eq:ito-diff}
\dd X_t = u(t,X_t)\dd t + \sqrt{2} \dd W_t.
\end{equation}
This admits a unique strong solution due to stochastic regularization, see~\cite{Ve81,Ve83,Da07}. In complete generality, the drift can either substantially increase or decrease the variance of $X_t$ compared to the case of no drift given by $\sqrt{2} W_t$. The clearest examples are given by drifts with non-trivial divergence. In particular, consider
\[u^{\mathrm{c}}(x) = - \frac{x}{|x|} \text{ or } u^{\mathrm{r}}(x) = \frac{x}{|x|}.\]
Taking $u(t,x) = u^{\mathrm{c}}(x)$ and $X_0=0$, we see that the law of $X_t$ converges to a stationary measure which has all moments. As such, $\Var(X_t) \to C <\infty$, so the drift strongly suppress variance growth of $X_t$ at long times compared to $\Var(\sqrt{2} W_t) =2dt.$ On the other hand, taking $u(t,x) = u^{\mathrm{r}}(x)$, we have at long times $\Var(X_t)\approx t^2 \gg \Var(\sqrt{2} W_t)$.

Thus we see that drifts with negative divergence tend to be concentrating and suppress variance growth while drifts with positive divergence tend to be spreading and enhance variance growth. It is thus natural to ask the effect of divergence-free drifts. In this case, the dynamics are more subtle as we can no longer cause pure spreading solely due to the drift: the flow map of the drift is volume-preserving. However, it is certainly still possible to have enhancement of variance growth due to the presence of a divergence-free drift. In 2D, taking $u(t,x) = (0,\lambda \sin(x_1))$, at long times, $\Var(X_t) \approx \lambda^2 t \gg \Var(\sqrt{2} W_t)$ provided $\lambda$ is sufficiently large. 

In fact, the variance of the process defined by~\eqref{eq:ito-diff} can be quantitatively related to the dissipation of the advection-diffusion operator $\solP^u_t$ by a fluctuation-dissipation relation~\cite{drivas_lagrangian_2017,drivas_lagrangian_2017-1}. This connection, together with the variance growth enhancing properties of divergence-free flows, has been used to study enhanced dissipation~\cite{zelati_stochastic_2021} as well as anomalous dissipation~\cite{johansson_anomalous_2024}. In the other direction, there are myriad examples of enhanced dissipation~\cite{constantin_diffusion_2008,bedrossian_enhanced_2017,feng_dissipation_2019,zelati_relation_2020,albritton_enhanced_2022,coti_zelati_enhanced_2023,villringer_enhanced_2024,bedrossian_almost-sure_2021,cooperman_harris_2024} and anomalous dissipation~\cite{drivas_anomalous_2022,colombo_anomalous_2023,armstrong_anomalous_2025,burczak_anomalous_2023,elgindi_norm_2024,rowan_anomalous_2024,agresti_anomalous_2024,hess-childs_universal_2025} that are proven using other techniques but through the fluctuation-dissipation relation provide examples of divergence-free flows that enhance the variance growth of $X_t$. The capacity for divergence-free vector fields to enhance the variance growth of $X_t$ has thus been well utilized. We are then motivated to ask: \textit{can a divergence-free drift \textbf{decrease} the variance of $X_t$} compared to $\sqrt{2} W_t$?

The initial answer to this question is \textit{yes, in the case that $X_0$ has a nondeterministic distribution}. In that case, we naturally compare $X_t$ to the evolution under pure diffusion, given by $X_0 + \sqrt{2}W_t$. In particular, consider taking the law of $X_0$ to be $\frac{1}{2}\delta_{(-R,0)}+\frac{1}{2}\delta_{(R,0)}$ for some large $R$. Then consider some divergence-free truncation $\tilde u(x,y)$ of 
\[u(x) = (-x_1, x_2),\]
so that $\tilde u$ is bounded and agrees with $u$ on $B_{2R}$. For short times, $\tilde u$ contracts $X_t$ towards the origin, decreasing the variance, before the action of $\tilde u$ in the second coordinate begins to increase the variance again.

However, in the case that~$X_0$ is deterministic---or as is more generally considered in this paper, $X_0$ has a symmetric decreasing law---the possibility of the drift causing the variance of~$X_t$ to be smaller than the variance of the pure diffusive process $X_0 + \sqrt{2}W_t$ is less clear. The answer provided in this work is that \textit{a divergence-free drift always increases the variance of $X_t$} compared to $\sqrt{2}W_t +X_0$, given the correct assumptions on the law of~$X_0.$

We further show that all the $L^p$ norms of the law of $X_t$---which is described by $\solP^u_t$ applied to the law of $X_0$---are less than the $L^p$ norms of the law of pure diffusive process $X_0 + \sqrt{2}W_t$. Additionally, we show that the entropy of $X_t$ is greater than that of $X_0 + \sqrt{2} W_t$. These are other manifestations of the loose statement \textit{divergence-free drifts always make the law of~$X_t$ ``less concentrated''.}

This imprecise sentiment is given its most precise form in Theorem~\ref{thm:main-result}, as the modulus of continuity $C_{\mathrm{Law}(X_t)}(\alpha)$ contains effectively all the information on how concentrated the law of $X_t$ is, though the statements of Corollary~\ref{cor:main-cor} are perhaps clearer.

There is an interesting comparison between this line of inquiry and that pursued in~\cite{miles_diffusion-limited_2018}, in which they consider the extent to which diffusion can suppress the mixing effect of the advective term. In general, mixing is strongly suppressed on very small length scales due to the smoothing effect of the diffusion; one version of this phenomenon is proven in~\cite{hairer_lower_2024}. We are however pursuing the opposite question: can (divergence-free) advection suppress the diffusion? And we provide the negative answer: under mild conditions, it cannot.

Finally, we remark on Theorem~\ref{thm:torus}, which says that when we take our domain to be $\T^d$ instead of $\R^d$, we can have divergence-free advection in some sense suppress the diffusion. In particular, there exists a flow for which the fundamental solution to the advection-diffusion equation has larger $L^2$ norms at long times when compared to the $L^2$ norm of the heat kernel. The mechanism for this is that rotation symmetries are broken on the torus. In particular, the level sets of the heat kernel at long times are highly non-circular. Thus there exists divergence-free flows which decrease the measure of the level sets (while necessarily preserving the measure of the super-level sets). By flowing the super-level sets toward being balls, this advecting flow will decrease the effect of diffusion at future times---essentially due to the isoperimetric inequality. We give an explicit example that exhibits this that takes advantage of the exact behavior of Fourier modes of the solution. We note however that, for example by using PDE arguments, one can still show inequalities similar to those of Corollary~\ref{cor:main-cor} when the domain is~$\T^d$ instead of~$\R^d$, it is just the inequalities \textit{will no longer be true without a prefactor constant.}

\subsection{Symmetric decreasing rearrangements}

For the convenience of the reader, we recall the definition of a symmetric decreasing rearrangement as well as the pivotal \textit{Riesz rearrangement inequality}. Before we do this, we recall that for an arbitrary measurable nonnegative function $f: \R^d \to [0,\infty)$, the following \emph{layer-cake} representation holds:
\[f(x) = \int_0^\infty \indc_{\{f > t\}}(x)\dd t\]
for all $x\in\R^d$.

\begin{definition}
    For a Borel set $E \subseteq \R^d$, the \textit{symmetric rearrangement of $E$}, denoted $E^*$, is the set
    \[E^* := B_r \text{ where } |B_r| = |E|.\]
    Then the \textit{symmetric decreasing rearrangement of $f$}, denoted $f^*$, is the function
    \[f^*(x) := \int_0^\infty \indc_{\{f>t\}^*}(x)\dd t.\]
    For a general positive finite measure $\mu \in \mathcal{M}_+(\R^d)$ with Lebesgue decomposition
    $\mu(\dd x) = f(x)\dd x + \beta(\dd x)$,
    we define the \textit{symmetric decreasing rearrangement of $\mu$}, denoted $\mu^*$, by
    \[\mu^*(\dd x) := f^*(x)\dd x + \beta(\R^d)\delta_0(\dd x).\]
\end{definition}

We note that $f^*$ is defined precisely so that $f^*$ is symmetric decreasing in the sense of Definition~\ref{def:symmetric-decreasing} and
$|\{f>t\}| = |\{f^*>t\}|$ for all $t>0$. We also note---as is made clear by Lemma~\ref{lem:concentration-representation}---that the symmetric rearrangement gives the minimal symmetric decreasing measure $\mu^*$ such that $\mu \preceq \mu^*$.   
The main result we will need about symmetric decreasing rearrangements is the following (for a detailed exposition, see~\cite[Chapter 3]{lieb_analysis_2001}).

\begin{theorem}[Riesz rearrangement inequality]
\label{thm:riesz-rearrangement}
    Let $f,g,h : \R^d \to [0,\infty)$ be measurable functions with $f,g,h \in L^1(\R^d)$. Then 
    \begin{equation}
        \label{eq:riesz-ineq}
    \iint f(x) g(x-y) h(y)\dd x \dd y \leq \iint f^*(x) g^*(x-y) h^*(y)\dd x \dd y.
    \end{equation}
    As a limiting case,
    \begin{equation}
        \label{eq:rearrange-test}
    \int f(x) g(x)\dd x \leq \int f^*(x)g^*(x)\dd x.
    \end{equation}
\end{theorem}

\subsection{Related literature on PDE and rearrangements}
\label{ss:related}

Symmetric rearrangements are a widely used tool for proving sharp estimates for elliptic and parabolic equations. In the seminal work of Talenti~\cite{talenti_elliptic_1976}, it was shown that if $-\Delta u=f$ and $-\Delta v=f^*$ then $u^*\leq v$. This form of a comparison principle has been generalized considerably to a variety of elliptic equations~\cite{talenti_nonlinear_1979,alvino_convex_1997,alvino_talenti_2023}, as well as analogous comparison principles for parabolic equations~\cite{bandle_symmetrizations_1976,de_bonis_symmetrization_2024,ferone_symmetrization_2024}. The latter often give estimates on the same modulus of absolute continuity used here.

Of particular relevance to the present work are~\cite{alvino_comparison_1990} and~\cite{diaz_steiner_2015}. In the former, inequalities are given between solutions of very general parabolic equations and appropriately symmetrized equations. Although similar to the main result of this paper, they do not specifically consider equations with divergence-free drifts. For this reason, their results must handle the most ``pessimistic'' case in which the drift generates substantial mass. This is impossible in the divergence-free setting, and thus their results do not meaningfully overlap with ours. In~\cite{diaz_steiner_2015}, they prove inequalities for concave semilinear parabolic equations. Their proof proceeds similarly to ours, using the Trotter product formula to decompose the parabolic operator into products of purely diffusive and purely absorptive operators, however, their setting and motivation are different.

Rearrangements also play a fundamental role in the study of two-dimensional fluid equations. Building on the seminal works of Arnold~\cite{arnold_sur_1966,arnold_sur_1966-1} and Benjamin~\cite{benjamin_alliance_1976}, Burton~\cite{burton_rearrangements_1987,burton_steady_1988} established the existence of steady-state solutions to the two-dimensional Euler equations by extremizing energy over classes of rearrangements. These methods have also been instrumental in proving stability results for such solutions~\cite{burton_global_2005,burton_nonlinear_2013,burton_compactness_2021}. More recently, similar ideas have been extended to analyze the stability of the two-dimensional Navier–-Stokes equation as a perturbation of Euler flows, see~\cite{gallay_arnolds_2024}.

Finally, in~\cite{jin_sharp_2024}, the authors give \textit{a priori} bounds on the modulus of absolute continuity for the vorticity of solutions to the two-dimensional Navier--Stokes equation with external forcing. The tools used for these estimates are similar to what is used in the present work, again using operator splitting. However, their primary motivation is to establish sharp conditions for energy balance in the presence of external forces rather than comparing solutions, leading to different conclusions.

\subsection{Overview of the argument}
\label{ss:overview}

There are two main ideas used in the argument of Theorem~\ref{thm:main-result}. The first is to study the modulus of absolute continuity with respect to the Lebesgue measure~\eqref{eq:C-mu-def}. Our primary motivation for this problem is to understand the behavior of variances and $L^p$ norms. However, directly studying these objects is less straightforward than studying the modulus of absolute continuity. In Section~\ref{sec:proof-of-cor}, we show how our statement about the modulus of absolute continuity directly gives the relevant facts about variances, entropies, and $L^p$ norms.

The second idea is to use operator splitting to study the evolution of the modulus of absolute continuity $C_\mu$. That is, for suitably regular advecting flows $u$, we can approximate the advection-diffusion operator by a pulsed diffusion, in which the advective and diffusive steps are alternated. The advantage of splitting the operator in this way is that, since the flow is divergence-free, the advective step \textit{leaves the modulus of continuity unchanged}. As such, we only need that the diffusive step is order-preserving under the relation $\preceq,$ which is provided by Proposition~\ref{prop:heat-increases-diffuseness}. It is there that we are essentially using that the measure we are comparing to, $\nu$, is symmetric decreasing and thus so is $\e^{t\Delta} \nu$ for $t>0.$ Iterating these two simple observations allows us to conclude a version of Theorem~\ref{thm:main-result} for the split operator and a special class of advecting flows. Using that the split operator converges to the advection-diffusion operator, and then that the general advecting flows are well-approximated by the special class of advecting flows, we are then able to conclude the general statement. 

Before proceeding to the proofs, we remark on the strictness of the inequality in Theorem~\ref{thm:main-result}. We have, under suitable assumptions, that for all $\alpha \geq 0$,
\[C_{\solP^u_t \mu}(\alpha) \leq C_{\e^{t\Delta}\nu}(\alpha).\]
It is natural however to ask if, under suitable non-degeneracy assumptions on the advecting flow~$u$, one can in fact make this inequality strict. That is, that there must exist some $\alpha>0$ such that
\[C_{\solP^u_t\mu}(\alpha)<C_{e^{t\Delta}\nu}(\alpha).\]
Upon inspection of the proof of Theorem~\ref{thm:main-result}, one can see that if $C_\mu(\alpha)<C_\nu(\alpha)$ for \textit{any} $\alpha\geq 0$, then at future times $t$
\[\forall\alpha>0,\quad C_{\solP_t^u\mu}(\alpha)<C_{e^{t\Delta}\nu}(\alpha).\] Thus, the only way for strictness to fail at some future time $t$ is if it also fails for all previous times. The natural conjecture is that strictness fails up to time $t$ if and only if $\mu$ is a translate of $\nu$ and $u$ exclusively generates solid body transformations.

\begin{conjecture}\label{conjecture}
    Suppose $u$ is a divergence-free vector field, $\mu\preceq \nu$, and $\nu$ is symmetric decreasing. Then, the following are equivalent: 
    \begin{enumerate}
        \item\label{item:i} we have $e^{s\Delta}\nu\preceq\solP_s^u\mu$ for all $s\in[0,t]$;
        \item\label{item:ii} $\mu$ is a translate of $\nu$, and there exists $v:[0,t]\rightarrow\R^d$ and $A:[0,t]\rightarrow \R^{d\times d}$ with the property that $u(s,x)=v(s)+A(s)x$ and $A(s)^\top = - A(s)$ for all $s\in[0,t]$ and $x\in\R^d$.
    \end{enumerate}
\end{conjecture}

While it is straightforward to prove that Item~\ref{item:ii} implies Item~\ref{item:i}, our techniques do not imply the reverse implication. The central step in the argument where the inequality can become strict is in our application of the Riesz rearrangement inequality in the proof that pure diffusion preserves the ordering $\preceq$ between the two solutions. It is well known that the Riesz rearrangement inequality is strict except in very special cases, as is shown by~\cite[Theorem 3.9]{lieb_analysis_2001} and \cite{burchard_cases_1996}.

Using these results one could conclude that---for suitably non-degenerate advecting flows---the pulsed diffusion approximation to the advection-diffusion equation does in fact have a strict inequality: that is in the notation of Subsection~\ref{ss:time-independent}, for all $\delta>0$
\[\solP^{u,\delta}_t \mu \preceq \e^{t \Delta}\nu\quad \text{and} \quad \e^{t \Delta}\nu \not \preceq\solP^{u,\delta}_t \mu.\]
However, without a much more quantitative understanding of the discrepancy in the Riesz rearrangement inequality, it is not at all clear how to maintain the strictness of this inequality in the $\delta \to 0$ limit that we take to access the dynamics of the advection-diffusion operator~$\solP^u_t.$ As such, we leave Conjecture~\ref{conjecture} open.

\section{Proof of Theorem~\ref{thm:main-result}}

We note first the following stability in $f$ of the modulus of absolute continuity $C_f$.

\begin{lemma}
    \label{lem:concentration-continuity}
    Suppose $f,g \in L^1_+(\R^d)$, then for all $p \in [1,\infty]$ and $\alpha \geq 0$, we have
    \[|C_f(\alpha) - C_g(\alpha)| \leq \alpha^{1-\frac 1p} \|f-g\|_{L^p}.\]
\end{lemma}

\begin{proof}
    Let $p \in [1,\infty]$ be arbitrary and let $q$ be its H\"older conjugate. Fix $\alpha > 0$ and, without loss of generality, suppose $C_f(\alpha) \geq C_g(\alpha)$. Then, by H\"older's inequality and the definition of~$C_g(\alpha)$, we have, for every Borel set $E$ with $|E|=\alpha$, that
    \[\int_E f(x)\dd x = \int_E f(x)-g(x) \dd x + \int_E g(x) \dd x \leq |E|^{\frac 1q} \|f-g\|_{L^p} + C_g(\alpha) = \alpha^{\frac 1q} \|f-g\|_{L^p} + C_g(\alpha).\]
    Taking the supremum over $|E| = \alpha$, we conclude that $C_f(\alpha) \leq \alpha^{\frac 1q} \|f-g\|_{L^p} + C_g(\alpha)$.
\end{proof}

\subsection{Proof of the time-independent case by operator splitting}

\label{ss:time-independent}

Throughout this subsection, fix $u \in W^{1,\infty}(\R^d)$ to be a time-independent divergence-free vector field. Our goal is to show the following special case of Theorem~\ref{thm:main-result} for time-independent, spatially Lipschitz vector fields and $L^2$ initial data.

\begin{proposition}
\label{prop:l2-time-indepenent-smooth}
     If $f,g$ are such that
     \[f,g \in L^2(\R^d) \cap L^1_+(\R^d),\quad g^* = g, \quad \text{and} \quad f \preceq g,\]
     then 
     \[\solP^u_t f \preceq \e^{t\Delta} g\]
     for all $t\geq 0$.
\end{proposition}

\begin{definition}
    We let $\e^{-t u \cdot \nabla}$ denote the solution semigroup on $L^2(\R^d)$ to the transport equation
    \[\partial_t \theta + u \cdot \nabla \theta =0.\]
\end{definition}

We then have the following straightforward proposition, whose proof is deferred to Subsection~\ref{ss:proof-heat-diffuses}, which gives that the advection semigroup leaves the modulus of absolute continuity unaffected.

\begin{proposition}
\label{prop:advection-maintains-concentration}
    If $f \in  L^2(\R^d) \cap L^1_+(\R^d)$, then
    \[C_f = C_{\e^{- tu \cdot \nabla}f}\]
    for all $t\geq 0$.
\end{proposition}

The second pivotal estimate we have is the following, which gives that the heat semigroup is order-preserving under the relation $\preceq$, provided that the measure on the right-hand side is symmetric decreasing. We also defer the proof to Subsection~\ref{ss:proof-heat-diffuses}.

\begin{proposition}
    \label{prop:heat-increases-diffuseness}
    Let $\mu, \nu \in \mathcal{M}_+(\R^d)$. If $\nu$ is symmetric decreasing and $\mu \preceq \nu$, then 
    \[\e^{t \Delta} \mu \preceq \e^{t \Delta} \nu\]
    for all $t\geq 0$.
\end{proposition}

Our strategy is to iterate the above two propositions for a split operator approximation to $\solP^u_t$ in order to prove Proposition~\ref{prop:l2-time-indepenent-smooth}.

\begin{definition}
    For any $\delta>0$ and $t\geq 0$, define the $\delta$-split operator approximation to $\solP^u_t$, denoted by $\solP^{u,\delta}_t,$ by
    \[\solP^{u,\delta}_t := \e^{(t - \delta\lfloor t/\delta\rfloor) \Delta}\big(\e^{-\delta u \cdot \nabla } \e^{\delta \Delta}\big)^{\lfloor t/\delta\rfloor}.\]
\end{definition}

\begin{remark}
    The term $\e^{(t - \delta\lfloor t/\delta\rfloor) \Delta}$, which does not typically appear in such splittings, is inserted---with no effect on the limit $\delta\to 0$---so that $\solP^{u,\delta}_t$ involves exactly the same exposure to pure diffusion as~$\e^{t\Delta}$.
\end{remark}

We will need the following convergence result, which follows from an application of the Trotter product formula; see~\cite{Tr59} or \cite[\S{III.5}]{engel_one-parameter_2000}.

\begin{proposition}
\label{prop:trotter}
    If $f \in L^2(\R^d)$, then  
    \[ 
    \lim_{\delta\to 0} \|\solP^{u,\delta}_t f - \solP^{u}_tf\|_{L^2} = 0
    \]
    for all $t \geq 0$.
\end{proposition}

We now state Proposition~\ref{prop:l2-time-indepenent-smooth} with the split operator $\solP^{u,\delta}_t$ in place of $\solP^u_t$.

\begin{proposition}
\label{prop:split-operator-increases-diffuseness}
     If $f,g$ are such that
     \[f,g \in L^2(\R^d) \cap L^1_+(\R^d),\quad g^* = g, \quad \text{and} \quad f \preceq g,\]
     then
     \[\solP^{u,\delta}_t f \preceq \e^{t\Delta} g\]
     for all $t\geq 0$ and $\delta>0$.
\end{proposition}

Combining the above results, this then directly gives Proposition~\ref{prop:l2-time-indepenent-smooth}. 
\begin{proof}[Proof of Proposition~\ref{prop:l2-time-indepenent-smooth}]
    Direct from Proposition~\ref{prop:split-operator-increases-diffuseness} and then taking $\delta \to 0$ and using Proposition~\ref{prop:trotter} and Lemma~\ref{lem:concentration-continuity}.
\end{proof}

\begin{proof}[Proof of Proposition~\ref{prop:split-operator-increases-diffuseness}]
    Fix $f$, $g$, and $\delta$ as above. We prove the following statements---indexed by $n \in \N$---by induction:
    \[\text{for all } t \in [0,(n+1)\delta),\quad \solP_t^{u,\delta} f \preceq \e^{t \Delta}g.\]
    For $n=0$, this follows directly from the definition of $\solP^{u,\delta}_t$, the assumptions on $f$ and $g$, and Proposition~\ref{prop:heat-increases-diffuseness}. Inductively supposing the result holds for $n-1$, we want to show that 
    \[\text{for } t \in [n \delta, (n+1)\delta),\quad  \solP_t^{u,\delta} f \preceq \e^{t \Delta}g.\]
    If we have that $\solP_{n\delta}^{u,\delta} f \preceq \e^{n\delta \Delta}g,$ we can conclude again using Proposition~\ref{prop:heat-increases-diffuseness} for $t \in (n\delta, (n+1)\delta)$. Finally, for $t = n\delta$, we note that, by the definition of $\solP_t^{u,\delta}$,
    \[\solP_{n\delta}^{u,\delta} f = \lim_{t \nearrow n\delta} \e^{-\delta u \cdot \nabla} \solP_{t}^{u,\delta} f,\]
    with convergence in $L^2$. Then we have by the inductive hypothesis that $\solP_t^{u,\delta} f \preceq \e^{t \Delta} g$ for $t < n \delta$. Then by Proposition~\ref{prop:advection-maintains-concentration}, we have that for all $t<n\delta$,
    \[\e^{-\delta u \cdot \nabla} \solP_{t}^{u,\delta} f \preceq \e^{t \Delta}g.\]
    Passing to the limit and using Lemma~\ref{lem:concentration-continuity}, we conclude.
\end{proof}

\subsection{\texorpdfstring{Reduction to the case of $L^2$ data and a smooth time-independent vector field}{Reduction to the case of L2 data and a smooth time-independent vector field}}

We now reduce the general statement of Theorem~\ref{thm:main-result} to the special case given by Proposition~\ref{prop:l2-time-indepenent-smooth} by approximation. To this end, we show that any bounded, measurable, divergence-free vector field $u$ can be suitably approximated---while getting convergence of the associated solutions up to a fixed time---by those we can treat using Proposition~\ref{prop:l2-time-indepenent-smooth}, i.e., those that are piecewise constant in time and uniformly Lipschitz in space.

\begin{lemma}
    \label{lem:approximate-by-mollification}
    Let $u \in L^\infty_{t,x}$ be a divergence-free vector field and let $g \in L^2(\R^d)$. Then there exists a sequence $(u_j)_j$ of bounded, divergence-free, piecewise constant in time, uniformly Lipschitz in space vector fields such that 
        \[\lim_{j\to\infty}\|\solP^{u_j}_t g - \solP^u_t g\|_{L^2} = 0.\]
    for all $t >0$.
\end{lemma}

\begin{proof}
    Extend $u$ to $\R \times \R^d$ by $0$ for $t <0.$ With $(\eta_\varepsilon)_\varepsilon$ a standard family of space-time mollifiers, let 
    $u_\varepsilon := \eta_\varepsilon * u$ and $u_{\varepsilon,\delta}(t,x):= u_\varepsilon(\lfloor t/\delta \rfloor \delta, x)$.
    In $L^p_{\text{loc}}(\R\times\R^d)$, we have $u_\varepsilon \to u$ as $\varepsilon \to 0$, and $u_{\varepsilon,\delta} \to  u_{\varepsilon}$ as $\delta \to 0$ with $\varepsilon$ fixed. 
    Hence, we can extract a sequence $(\varepsilon_j,\delta_j)_j$ with the property that, still in $L^p_{\text{loc}}(\R\times\R^d)$, we have
    $u_j := u_{\varepsilon_j,\delta_j}  \to u$ as $j\to\infty$.
    Therefore, there exists a subsequence along which we have almost everywhere convergence; for notational simplicity, let us simply say that $u_j \to u$ almost everywhere as $j\to\infty$. 
    
    By construction, each $u_j$ is divergence-free, piecewise constant in time and uniformly Lipschitz in space, and it remains to show that, for each $t>0$,
    \[\lim_{j\to\infty}\|\solP^{u_j}_t g - \solP^u_t g\|_{L^2} = 0.\]
    To this end, we define $\gamma_j(t,\,\cdot\,) := \solP^{u_j}_t g - \solP^u_t g$ and $\varphi(t,\,\cdot\,) := \solP^u_t g$ and note that, by direct computation, we have $\gamma_j(0,\,\cdot\,) =0 $ and
    \[\partial_t \gamma_j -\Delta \gamma_j + u_j \cdot \nabla \gamma_j = (u-u_j) \cdot \nabla \varphi.\]
    Thus, using that $\nabla \cdot u_j = \nabla \cdot u = 0$ and Young's inequality,
    \begin{align*}
        \frac{\dd}{\dd t} \|\gamma_j(t,\,\cdot\,)\|_{L^2}^2 
            &= - 2\|\nabla \gamma_j(t,\,\cdot\,)\|_{L^2}^2 - 2\int \nabla \gamma_j(t,x) \cdot (u(t,x)-u_j(t,x)) \varphi(t,x)\dd x \\
            &\leq  \int |u(t,x)-u_j(t,x)|^2 \varphi(t,x)^2\dd x,
    \end{align*}
    and therefore
    \[\|\gamma_j(t,\,\cdot\,)\|_{L^2}^2 \leq \int_0^t \int |u(s,x)-u_j(s,x)|^2 \varphi(s,x)^2\dd x\dd s.\]
    We have already seen that $u-u_j \to 0$ almost everywhere as $j\to\infty$, and thus the same is true of $|u-u_j|^2\varphi^2$. In addition, we have $\varphi \in L^2([0,t]\times\R^d)$ and $|u-u_j| \leq 2 \|u\|_{L^\infty_{t,x}}$, so $|u-u_j|^2\varphi^2$ is dominated by 
        %$2\|u\|^2_{L^\infty_{t,x}}\varphi^2 \in L^1([0,t]\times\R^d)$
        an element of $L^1([0,t]\times\R^d)$.
    By dominated convergence, this implies
    \[\lim_{j\to\infty} \|\gamma_j(t,\,\cdot\,)\|_{L^2} =0. \]
    Recalling the definition of $\gamma_j$, we conclude.
\end{proof}

The following is now direct from Proposition~\ref{prop:l2-time-indepenent-smooth}, Lemma~\ref{lem:approximate-by-mollification}, and Lemma~\ref{lem:concentration-continuity}. 

\begin{proposition}
    \label{prop:l2-general}
     Let $u \in L^\infty_{t,x}$ be a divergence-free vector field.
     If $f,g$ are such that
     \[f,g \in L^2(\R^d) \cap L^1_+(\R^d),\quad g^* = g, \quad \text{and} \quad f \preceq g,\]
     then 
     \[\solP^u_t f \preceq \e^{t\Delta} g\]
     for all $t\geq 0$.
\end{proposition}

Finally, we remove the requirement that the data is $L^2(\R^d)$ to conclude.

\begin{proof}[Proof of Theorem~\ref{thm:main-result}]
    Let
    \[\mu_\varepsilon := \e^{\varepsilon \Delta}\mu \quad \text{and}\quad \nu_\varepsilon := \e^{\varepsilon \Delta}\nu.\]
    Then by Proposition~\ref{prop:heat-increases-diffuseness}, we have that for $\varepsilon>0,$
    \[\mu_\varepsilon \preceq \nu_\varepsilon,\]
    and by Young's convolution inequality $\mu_\varepsilon, \nu_\varepsilon \in L^2(\R^d)$. Thus by Proposition~\ref{prop:l2-general}, 
    \[\solP^u_t \mu_\varepsilon \preceq \e^{t\Delta} \nu_\varepsilon.\]
    By standard parabolic smoothing estimates, for fixed $t>0$, we have that in $L^2(\R^d)$
    \[\solP^u_t \mu_\varepsilon \stackrel{\varepsilon \to0}{\to} \solP^u_t \mu\quad\text{and}\quad \e^{t\Delta}\nu_\varepsilon \stackrel{\varepsilon \to0}{\to} \e^{t\Delta} \nu.\]
    Thus using Lemma~\ref{lem:concentration-continuity}, we can pass to the limit and conclude.
\end{proof}

\subsection{Proofs of Propositions~\ref{prop:advection-maintains-concentration} and~\ref{prop:heat-increases-diffuseness}}
\label{ss:proof-heat-diffuses}

We first verify the simple fact that rearranging does not affect the modulus of absolute continuity.

\begin{lemma}
    \label{lem:concentration-representation} 
    For every $\mu \in \mathcal{M}_+(\R^d)$, we have 
    \[C_\mu = C_{\mu^*}.\]
\end{lemma}
\begin{proof}
    The definition of $\mu^*$ allows us to readily reduce to the case that $\mu(\dd x) = f(x)\dd x$ for some $f \in L^1_+$. In which case, the statement becomes that
    \[
        \sup_{|E| =\alpha} \int_E f(x) \dd x =\int_{B_{r_\alpha}} f^*(x) \dd x,
    \]
    where $r_\alpha:=(|B_1|^{-1}\alpha)^{1/d}.$ Note that by~\eqref{eq:rearrange-test}, we have that for $|E| = \alpha$ 
    \[\int_E f(x)\dd x \leq \int_{E^*} f^*(x)\dd x = \int_{B_{r_\alpha}} f^*(x)\dd x.\]
    On the other hand,
    \[
    \int_{B_{r_\alpha}} f^*(x)\dd x 
    = f^*(r_\alpha)\big|\{f^* = f^*(r_\alpha)\} \cap B_{r_\alpha}\big| 
    +\int_{\{f^*> f^*(r_\alpha)\}} f^*(x)\dd x.
    \]
    whereby $f^*(r_\alpha)$ we mean $f^*(r_\alpha,0,\dots,0).$  
    One can then directly see by the layer-cake representation,
    \[\int_{\{f^*> f^*(r_\alpha)\}} f^*(x)\dd x = \int_{\{f>f^*(r_\alpha)\}} f(x)\dd x.\]
    Thus if $|\{f^* = f^*(r_\alpha)\}| = 0$, we can conclude.  Otherwise 
    \[|\{f^* = f^*(r_\alpha)\}| = |\{f = f^*(r_\alpha)\}| \ne 0.\]
    In this case, define
    \[L_{\alpha,R} := \{f = f^*(r_\alpha)\} \cap B_R.\]
    Note that $|L_{\alpha, 0}| =0, |L_{\alpha,\infty}| =| \{f = f^*(r_\alpha)\}|,$ and $|L_{\alpha,R}|$ is continuous in $R$. Thus we can choose $R_\alpha$ such that 
    \[|L_{\alpha,R_\alpha}| = \big|\{f^* = f^*(r_\alpha)\}\cap B_{r_\alpha}\big|.\]
    Then we have that 
    \[\int_{B_{r_\alpha}} f^*(x)\dd x = f^*(r_\alpha) |L_{\alpha,R_\alpha}| + \int_{\{f> f^*(r_\alpha)\}} f(x)\dd x  = \int_{\{f> f^*(r_\alpha)\} \cup L_{\alpha,R_\alpha}} f(x)\dd x,\]
    and
    \[|\{f> f^*(r_\alpha)\} \cup L_{\alpha,R_\alpha}| = |\{f^* > f^*(r_\alpha)\}| + \big|\{f^* = f^*(r_\alpha)\} \cap B_{r_\alpha}\big| = |B_{r_\alpha}| = \alpha,\]
    thus allowing us to conclude.
\end{proof}

We can now prove Proposition~\ref{prop:advection-maintains-concentration}.

\begin{proof}[Proof of Proposition~\ref{prop:advection-maintains-concentration}]
    Since $u$ is divergence-free, we have that $\big(\e^{-tu\cdot\nabla}f)^* = f^*$, thus
    \[C_f = C_{f^*} = C_{(\e^{-tu\cdot\nabla}f)^*} = C_{\e^{-tu\cdot\nabla}f},\]
    as claimed.
\end{proof}

We want to slightly generalize Theorem~\ref{thm:riesz-rearrangement} to allow $f$ and $h$ to be measures in order to prove Proposition~\ref{prop:heat-increases-diffuseness}.

\begin{corollary}
    Let $\mu,\nu \in \mathcal{M}_+(\R^d)$ and $g : \R^d \to [0,\infty)$ bounded and continuous. Then
    \begin{equation}
        \label{eq:riesz-measure}
    \iint g(x-y)\mu(\dd x) \nu(\dd y) \leq \iint g^*(x-y) \mu^*(\dd x)\nu^*(\dd y).
    \end{equation}
\end{corollary}

\begin{proof}
    Let $\mu,\nu$ have Lebesgue decompositions
    $\mu(\dd x) = f(x) \dd x + \alpha(\dd x)$ and $\nu(\dd x) = h(x)\dd x + \beta(\dd x)$.
    Then by~\eqref{eq:riesz-ineq} and then~\eqref{eq:rearrange-test},
    \begin{align*}
        \iint g(x-y)\mu(\dd x) \nu(\dd y) &= \iint g(x-y)f(x)h(y) \dd x \dd y + \iint g(x-y)\alpha(\dd x) h(y)\dd y 
        \\&\qquad+ \iint g(x-y)f(x)\dd x \beta(\dd y) + \iint g(x-y)\alpha(\dd x) \beta(\dd y)
        \\&\leq \iint g^*(x-y)f^*(x)h^*(y) \dd x \dd y + \alpha(\R^d) \sup_x \int g(x-y) h(y)\dd y  
        \\&\qquad+ \beta(\R^d) \sup_y \int g(x-y) f(x)\dd x + \alpha(\R^d) \beta(\R^d) \sup_x g(x)
        \\&\leq \iint g^*(x-y)f^*(x)h^*(y) \dd x \dd y + \alpha(\R^d)  \int g^*(y) h^*(y)\dd y  
        \\&\qquad+ \beta(\R^d)\int g^*(x) f^*(x)\dd x + \alpha(\R^d) \beta(\R^d) g^*(0)
        \\&= \iint g^*(x-y) \mu^*(\dd x) \nu^*(\dd y).
    \end{align*}
    This is~\eqref{eq:riesz-measure}.
\end{proof}

We also need the following simple fact about the convolution of symmetric decreasing functions, whose proof is direct from the layer cake representation and the special case of the convolution of indicators on balls. 

\begin{lemma}
\label{lem:convolution-of-symm-dec}
    Suppose $f,g\in L^1_+(\R^d)$ and are symmetric decreasing. Then the convolution $f \star g$ is also symmetric decreasing.
\end{lemma}

% \begin{proof}
%     Using the layer-cake representation together with $f,g$ symmetric decreasing, we have that
%     \[f = \int_0^\infty \indc_{B_{r_f(t)}}\dd t \quad \text{and} \quad g = \int_0^\infty \indc_{B_{r_g(t)}} \dd t\]
%     for some decreasing functions $r_f,r_g : (0,\infty) \to [0,\infty).$ Then by linearity,
%     \[f \star g = \int_0^\infty \int_0^\infty \indc_{B_{r_f(t)}} \star \indc_{B_{r_g(s)}}\dd tds.\]
%     Thus since the sum of symmetric decreasing functions is symmetric decreasing, it suffices to prove that for any $r,s,$ $\indc_{B_r} \star \indc_{B_s}$ is symmetric decreasing. Then we have that
%     \[\indc_{B_r} \star \indc_{B_s}(x) = |B_r(0) \cap B_s(x)|,\]
%     which is readily seen to be symmetric decreasing.
% \end{proof}

The following is then direct from Lemma~\ref{lem:convolution-of-symm-dec} and that the heat kernel is symmetric decreasing.
\begin{lemma}
\label{lem:heat-preserves-symm-dec}
    If $\mu \in \mathcal{M}_+(\R^d)$ is symmetric decreasing, then $\e^{t\Delta} \mu$ is symmetric decreasing for all $t\geq 0$.
\end{lemma}

\begin{lemma}
\label{lem:test-with-monotone}
    Let $\mu,\nu \in \mathcal{M}_+(\R^d)$ be symmetric decreasing. Let $f: \R^d \to [0,\infty)$ be symmetric decreasing and let $g : \R^d \to [0,\infty)$ be symmetric \emph{increasing}, that is $-g$ is symmetric decreasing.
    Then if $\mu \preceq \nu$,
    \begin{equation}
        \label{eq:test-with-decreasing}
    \int f(x)\,\mu(\dd x) \leq \int f(x)\, \nu(\dd x).
    \end{equation}
    If additionally, $\mu(\R^d) = \nu(\R^d)$, then
    \begin{equation}
        \label{eq:test-with-increasing}
    \int g(x) \,\mu(\dd x) \geq \int g(x) \,\nu(\dd x).
    \end{equation}
\end{lemma}

\begin{proof}
    Note that, since $f$ is symmetric decreasing, there is some decreasing function $r : (0,\infty) \to [0,\infty)$ such that 
    \[f(x) = \int_0^\infty \indc_{B_{r(t)}}(x)\dd t.\]
    Then 
    \[
        \int f(x)\,\mu(\dd x)= \int_0^\infty  \mu(B_{r(t)})\dd t \leq \int_0^\infty \nu(B_{r(t)})\dd t = \int f(x)\,\nu(\dd x),\]
    where we use that $\mu \preceq \nu$ and that $\mu,\nu$ are symmetric decreasing for the middle inequality. This gives~\eqref{eq:test-with-decreasing}.

    For~\eqref{eq:test-with-increasing}, we suppose $\mu(\R^d) = \nu(\R^d)$. Then by $g$ symmetric increasing, we have that for some $r : (0,\infty) \to [0,\infty]$ increasing,
    \[g(x) = \int_0^\infty 1- \indc_{B_{r(t)}}(x)\dd t.\]
    Then
    \[\int g(x)\mu(\dd x) = \int_0^\infty \mu(\R^d) - \mu(B_{r(t)})\dd t \geq \int_0^\infty \nu(\R^d) - \nu(B_{r(t)})\dd t = \int g(x)\nu(\dd x),\]
    as desired.
\end{proof}

\begin{proof}[Proof of Proposition~\ref{prop:heat-increases-diffuseness}]
    Fix $\mu$ and $\nu$ as in the statement. Also fix $\alpha>0$ and let $r_\alpha = (|B_1|^{-1} \alpha)^{1/d}$. Then for any Borel set $E$ such that $|E|=\alpha$, we have that
    \[\int_E \e^{t\Delta}\mu(\dd x) = \iint \indc_E(x) \Phi_t(x-y) \mu(\dd y) \leq  \iint \indc_E^*(x) \Phi_t^*(x-y) \mu^*(\dd y) = \int \indc_{B_{r_\alpha}}\star \Phi_t(y) \mu^*(\dd y), \]
 where $\Phi_t$ is the heat kernel and we use the Riesz rearrangement inequality for measures~\eqref{eq:riesz-measure} for the inequality. Noting that $\mu^* \preceq \nu$ by Lemma~\ref{lem:concentration-representation}, Lemma~\ref{lem:convolution-of-symm-dec} gives that $\indc_{B_{r_\alpha}}\star \Phi_t$ is symmetric decreasing so Lemma~\ref{lem:test-with-monotone} gives 
 \[\int \indc_{B_{r_\alpha}}\star \Phi_t(y) \mu^*(\dd y) \leq \int \indc_{B_{r_\alpha}}\star \Phi_t(y) \nu(\dd y) = \int_{B_{r_\alpha}} \e^{t\Delta} \nu(\dd y) = C_{\e^{t\Delta} \nu}(\alpha).\]
Taking the supremum over $|E| = \alpha$, we conclude. 
\end{proof}

\section{Proof of Corollary~\ref{cor:main-cor}}
We first precisely specify what we mean by the variance $\Var$ and the entropy $H$.

\begin{definition}
\label{def:variance-entropy-def}
    For a probability measure $\mu \in \mathcal{M}_+(\R^d)$, define
    \[\Var(\mu) := \int |x|^2\,\mu(\dd x) - \Big(\int x\,\mu(\dd x)\Big)^2,\]
    with the convention that $\Var(\mu) = \infty$ if $\int |x|\,\mu(\dd x) = \infty$. For a probability density $f \in L^1_+(\R^d)$ with $f\log f \in L^1(\R^d)$, define
    \[ H(f) := - \int f\log f\dd x.\]
\end{definition}

\label{sec:proof-of-cor}

We now note two inequalities that will allow us to relate the relation $\preceq$ to $L^p$ norms, variance, and entropy. The following is a version of~\cite[249, 250]{hardy_inequalities_1952}; we will use this to prove the relation on $L^p$ norms and entropy.

\begin{proposition}
\label{prop:convex-dom}
    Suppose $F : [0,\infty) \to \R, F \in C^2(0,\infty)$ is convex, $F(0) = 0$, and $F'(0) > - \infty$. Let $f,g \in L^1_+(\R^d)$ be such that 
    $\int f(x)\dd x =\int g(x)\dd x$.
    If $f\preceq g$, then 
    \[\int F(f(x))\dd x \leq \int F(g(x))\dd x.\]
\end{proposition}

For the next proposition, we first need to define the \textit{symmetric increasing rearrangement.} The definition is motivated by the fact that, given a function $\varphi:\R^d\rightarrow[0,\infty)$, we have
\[\varphi(x) = \int_0^\infty 1 - \indc_{\{ \varphi \leq t\}}(x)\dd t.\]

\begin{definition}
For a function $\varphi:\R^d\rightarrow[0,\infty)$ with finite-measure sublevel sets---that is $|\{\varphi\leq t\}|<\infty$ for all $t\geq 0$---we define the \textit{symmetric increasing rearrangement of $\varphi$}, denoted~$\varphi^\vee$, by
\[\varphi^\vee (x) := \int_0^\infty 1-  \indc_{\{\varphi \leq t\}^*}(x)\dd t.\]
\end{definition}

The following is a version of~\cite[Theorem 368]{hardy_inequalities_1952}. We will use this to prove the relation on variances.

\begin{proposition}
\label{prop:test-with-increasing-rearrange}
    Suppose $\mu,\nu \in \mathcal{M}_+(\R^d)$ are probability measures such that $\nu$ is symmetric decreasing and $\mu \preceq \nu$. Let $\varphi : \R^d \to [0,\infty)$ have finite-measure sublevel sets. Then
    \[
    \int \varphi(x)\,\mu(\dd x) 
    \geq \int \varphi^\vee(x)\,\nu(\dd x).
    \]
\end{proposition}

We will provide the elementary proofs of Propositions~\ref{prop:convex-dom} and~\ref{prop:test-with-increasing-rearrange} at the end of this section. We first prove~\eqref{eq:Lp-var-ineqs}.

\begin{proof}[Proof of Corollary~\ref{cor:main-cor},~\eqref{eq:Lp-var-ineqs}]
    By Theorem~\ref{thm:main-result}, we have that $\solP_t^u \mu \preceq \e^{t\Delta} \nu$. Then the statement about $L^p$ norms for $p<\infty$ follows from Proposition~\ref{prop:convex-dom} applied with $F(t) = t^p$. The statement for $p =\infty$ is then obtained by taking a limit.

    For the inequality on the variances, we first note the inequality trivializes if
    $\int |x|\,\mu(\dd x) =\infty$.
    Otherwise, we apply Proposition~\ref{prop:test-with-increasing-rearrange} with
    \[\varphi(x) := \Big|x -\int y\,\solP_t^u\mu(\dd y)\Big|^2.\]
    Noting that $\varphi^\vee(x) = |x|^2$, we have that 
    \[\Var(\solP_t^u\mu) = \int \varphi(x)\, \solP_t^u\mu(\dd x) \geq \int \varphi^\vee(x)\,\e^{t\Delta}\nu(\dd x) = \Var(\e^{t\Delta}\nu),\]
    as claimed.
\end{proof}

We next prove~\eqref{eq:entropy-ineq}. The entropy is slightly more subtle as the integral $-\int f \log f$ is less well-behaved, potentially having both infinite negative and positive contributions. We first note that the condition of having a $\beta>0$ moment on the probability density, together with an $L^\infty$ bound, gives that the entropy density is absolutely integrable.

\begin{lemma}
\label{lem:conditions-for-entropy}
    Let $f$ be a probability density. If $f \in L^\infty(\R^d)$ and there exists $\beta > 0$ such that $\int |x|^\beta f(x) \dd x < \infty$, then $f \log f \in L^1(\R^d)$.
\end{lemma}

\begin{proof}
    We treat the positive and negative parts separately. For the positive part,
    \begin{align*}
        \int (f(x)\log f(x))_+ \dd x \leq \int f(x) \log (\|f\|_\infty \vee 1) \dd x = \log  (\|f\|_\infty \vee 1).
    \end{align*}
    As for the negative part, with $\beta$ as in the hypotheses of the lemma, by Jensen's inequality (for the concave function $t \mapsto \log t$) and H\"older's inequality,
    \begin{align*}
        \int (f(x)\log f(x))_- \dd x - \int |x|^\beta f(x) \dd x 
        &= \int f(x)[(\log f(x))_- + \log (\e^{-|x|^\beta})] \dd x\\
        % &= \int f(x)[-\log (f(x) \wedge 1) + \log (\e^{-|x|^\beta})] \dd x\\
        &= \int f(x)[\log ([f(x)^{-1} \vee 1] \e^{-|x|^\beta})] \dd x\\
        &\leq \log \int [1 \vee f(x)] \e^{-|x|^\beta} \dd x \\
        &\leq \log \int \e^{-|x|^\beta} \dd x + \log (\|f\|_\infty \vee 1),
    \end{align*}
    which is finite.
\end{proof}

We now show that, if $f,g \in L^1_+$ and they have absolutely integrable entropy densities, then $f \preceq g$ implies $H(f) \geq H(g).$

\begin{proposition}
\label{prop:S-Gamma-order}
    Let $f,g \in L^1_+$ be such that 
    $f\log f \in L^1$ and $g\log g\in L^1$.
    If $f \preceq g$, then $H(f) \geq H(g)$.
\end{proposition}

\begin{proof}
    For every $\varepsilon \in (0,1)$, the function $t \mapsto (t+\varepsilon)\log(t+\varepsilon) - \varepsilon\log\varepsilon$ satisfies the assumptions of Proposition~\ref{prop:convex-dom}. Hence, $f\preceq g$ implies that $H_\varepsilon(f) \geq H_\varepsilon(g)$, where
    \begin{align*}
        H_\varepsilon(f) := -\int_{\R^d} (f(x) + \varepsilon) \log (f(x)+\varepsilon) - \varepsilon\log\varepsilon \dd x.
    \end{align*} 
    Since 
    \[ 
        \lim_{\varepsilon\to0} (z + \varepsilon) \log (z+\varepsilon) - \varepsilon\log\varepsilon = z\log z
    \] 
    for every $z \geq 0$,
    it suffices to show that the limit $\varepsilon \to 0$ can be passed through the integrals defining $H_\varepsilon(f)$ and $H_\varepsilon(g)$ when $f$ and $g$ satisfy our assumptions. To this end, note that 
    \begin{align*}
        \left|(f(x) + \varepsilon) \log (f(x)+\varepsilon) - \varepsilon\log\varepsilon \right|
        &= f(x) \left|\log (f(x)+\varepsilon)\right| + \varepsilon\log\left(1 + \frac{f(x)}{\varepsilon}\right) \\
        &\leq f(x) (2\left|\log f(x)\right| + 2) + \varepsilon\left(\frac{f(x)}{\varepsilon}\right) \\
        &\leq 2f(x)|\log f(x)| + 3f(x)
    \end{align*}
    so we can conclude, under the assumptions, by the dominated convergence theorem.
\end{proof}

Finally, we can conclude~\eqref{eq:entropy-ineq}.

\begin{proof}[Proof of Corollary~\ref{cor:main-cor},~\eqref{eq:entropy-ineq}]
    By assumption, there is some $\beta>0$ such that $\int |x|^\beta d\mu(x) <\infty$. Now fix $t>0$ and let
    \[f := \solP^u_t \mu \quad \text{and}\quad g := \e^{t\Delta} \nu.\]
    By Theorem~\ref{thm:main-result}, we have that $f \preceq g$. Standard parabolic estimates, using that we have a $\beta$ moment on the initial data, then give
    \[\int |x|^\beta f(x)\dd x <\infty.\]
    Then by Proposition~\ref{prop:test-with-increasing-rearrange}, we have that
    \[\int |x|^\beta g(x)\dd x \leq \int |x|^\beta f(x)\dd x <\infty.\]
    We also have by standard parabolic estimates that $f,g \in L^\infty$. Thus by Lemma~\ref{lem:conditions-for-entropy}, we have that
        $f\log f \in L^1$ and $g\log g \in L^1$.
    Thus we conclude~\eqref{eq:entropy-ineq} by Proposition~\ref{prop:S-Gamma-order}. 
\end{proof}

\subsection{Proofs of Propositions~\ref{prop:convex-dom} and~\ref{prop:test-with-increasing-rearrange}}

We first introduce an object that is in some sense dual to $C_\mu(\alpha).$ 

\begin{definition}
    Given a finite, positive Borel measure $\mu$ we define the function
    \[\Gamma_\mu(h) := \int (f(x)-h)_+\dd x+\beta(\R^d),\quad h\geq 0\]
     where $\mu(\dd x)=f(x)\dd x+\beta(\dd x)$, is the Lebesgue decomposition of $\mu$ with respect to the Lebesgue measure.
\end{definition}

We now show that $C_\mu \leq C_\nu$ implies $\Gamma_\mu \leq \Gamma_\nu$.

\begin{lemma}\label{lem:equivalent_def}
    For any two positive Borel measures $\mu$ and $\nu$, if $\mu\preceq \nu$ then $\Gamma_\mu(h)\leq \Gamma_\nu(h)$ for all $h\geq 0.$
\end{lemma}

\begin{proof}
Suppose that $\mu(\dd x)=f(x)\dd x+\beta(\dd x)$ and $\nu(\dd x)=g(x)\dd x+\gamma(\dd x)$. Then Lemma~\ref{lem:concentration-representation} implies that since $\mu\preceq\nu$,
\begin{equation}\label{eq:ball_ineq}
\beta(\R^d)+\int_{B_r}f^*(x)\dd x\leq \gamma(\R^d)+\int_{B_r}g^*(x) \dd x
\end{equation}
for all $r\geq 0.$ Fix $h \geq 0$ and let $r(h)=\inf\{r\geq 0:f^*(r,0,\dots,0)\leq h\}$ so that
\[\int(f(x)-h)_+\dd x=\int(f^*(x)-h)_+\dd x=\int_{B_{r(h)}}(f^*(x)-h)\dd x.\]
Then
\begin{align*}
    \Gamma_\mu(h) 
    &=\beta(\R^d)+\int_{B_{r(h)}}(f^*(x)-h)\dd x \\
    &\leq\gamma(\R^d)+ \int_{B_{r(h)}}(g^*(x)-h)\dd x \\
    &\leq \gamma(\R^d)+\int_{B_{r(h)}}(g^*(x)-h)_+\dd x\\
    &\leq \gamma(\R^d)+\int(g^*(x)-h)_+\dd x\\
    &=\Gamma_\nu(h)
\end{align*}
where we've used~\eqref{eq:ball_ineq} in the first inequality. This completes the lemma.
\end{proof}

We can now use the inequality $\Gamma_f \leq \Gamma_g$ to straightforwardly conclude Proposition~\ref{prop:convex-dom}. 

\begin{proof}[Proof of Proposition~\ref{prop:convex-dom}]
    Note that since $F \in C^2$,
    \[F(x)= F'(0)x +\int_0^\infty F''(t) (x-t)_+\dd t,\]
    as can be verified by computing the second derivative of both sides. If $\int f(x) \dd x = \int g(x) \dd x$ and $f\preceq g$, then we have that 
    \begin{align*}
    \int F(f(x))\dd x &= F'(0)\int f(x)\dd x + \int_0^\infty F''(t)  \int (f(x)-t)_+ \dd x \dd t 
    \\&= F'(0)\int g(x)\dd x + \int_0^\infty F''(t) \Gamma_f(t)\dd t 
    \\&\leq F'(0)\int g(x)\dd x + \int_0^\infty F''(t) \Gamma_g(t)\dd t
    \\&= \int F(g(x))\dd x,
    \end{align*}
    where we use Lemma~\ref{lem:equivalent_def}. 
\end{proof}

For Proposition~\ref{prop:test-with-increasing-rearrange}, we need the following simple observation.

\begin{lemma}
\label{lem:Hardy-Littlewood-Polya}
    If $\mu \in\mathcal{M}^+$ and $\varphi : \R^d \to [0,\infty)$ has finite-measure sublevel set, then
    \[\int \varphi(x) \, \mu(\dd x) \geq \int \varphi^\vee(x) \, \mu^*(\dd x).\]
\end{lemma}
\begin{proof} 
    With the Lebesgue decomposition
    $\mu(\dd x) = f(x)\dd x + \beta(\dd x)$,
    we have
    \[\int \varphi(x)\, \mu(\dd x) \geq \int \varphi(x) f(x) \dd x + \inf_{x} \varphi(x) \beta(\R^d) = \int \varphi(x) f(x)\dd x + \beta(\R^d)\int \varphi^\vee(x) \delta_0(\dd x).\]
    Thus it suffices to prove the case that $\beta=0$. Then by~\eqref{eq:rearrange-test}, we have that for all $s,t\geq 0$
    \[ \int \indc_{\{f>t\}}(x)\indc_{\{\varphi\leq s\}}(x)\dd x\leq \int \indc_{\{f>t\}^*}(x)\indc_{\{\varphi\leq s\}^*}(x)\dd x.\]
    Using the layer cake representations of $\varphi$ and $f$ we thus find that
    \begin{align*}
    \int \varphi(x)f(x)\dd x&=\iiint \indc_{\{f>t\}}(x)(1-\indc_{\{\varphi\leq s\}}(x))\dd x\dd s\dd t
    \\&\geq \iiint \indc_{\{f>t\}^*}(x)(1-\indc_{\{\varphi\leq s\}}^*(x))\dd x\dd s\dd t 
    \\&=\int\varphi^\vee(x)f^*(x)\,\dd x
    \end{align*}
    as claimed.
\end{proof}

Proposition~\ref{prop:test-with-increasing-rearrange} is now direct.

\begin{proof}[Proof of Proposition~\ref{prop:test-with-increasing-rearrange}]
    By Lemma~\ref{lem:Hardy-Littlewood-Polya}, then Lemma~\ref{lem:test-with-monotone} using that $\varphi^\vee$ is symmetric increasing, $\mu \preceq \nu$, and $\nu$ is symmetric decreasing,
    \[\int \varphi(x)\, \mu(\dd x) \geq \int \varphi^\vee(x)\, \mu^*(\dd x) \geq \int \varphi^\vee(x)\, \nu(\dd x),\]
    as claimed.
\end{proof}

\section{Proof of Theorem~\ref{thm:torus}}

We now work on the $\T^d = [0,1]^d/\sim$ and for $f : \T^d\to \R$, let
\[\hat f(k) := \int \e^{-2\pi \mathrm{i} k \cdot x} f(x)\dd x.\]
We consider the vector field, defined in Fourier space by
\[\hat{u}^{T,h}(t,k) := \begin{cases}
     (-2\mathrm{i},\mathrm{i},0,\dotsc ,0)& k = (1,2,0,\dotsc ,0) \text{ and }T \leq t \leq T+h,\\
     (2\mathrm{i},-\mathrm{i},0,\dotsc ,0)& k =(-1,-2,0,\dotsc ,0) \text{ and } T \leq t \leq T+h,\\
     0 & \text{otherwise.}
\end{cases}
\]
We note that $u$ is compactly supported in Fourier space, $\hat{u}^{T,h}(t,-k) = \overline{\hat{u}^{T,h}(t,k) }$, and $$k \cdot \hat{u}^{T,h}(t,k) = 0,$$ thus ensuring $u$ is suitably smooth, real-valued and divergence-free.
We let $\varphi$ solve
\[\begin{cases}
    \partial_t \varphi - \Delta \varphi =0\\
    \varphi(0,\dd x) = \delta_0(\dd x),
\end{cases}\]
and for each $T,h>0$, and let $\psi^{T,h}$ solve
\[\begin{cases}
    \partial_t \psi^{T,h} - \Delta \psi^{T,h} + u^{T,h} \cdot \nabla \psi^{T,h} =0\\
    \psi^{T,h}(0,\dd x) = \delta_0(\dd x).
\end{cases}\]
We note that 
\[\psi^{T,h}(t,\,\cdot\,) = \varphi(t,\,\cdot\,) \quad \text{for } t \in [0,T].\]
We first state the following technical estimate. 

\begin{proposition}
\label{prop:torus}
    There exists $T_0>0$ such that for all $T\geq T_0$, there exists $h_0(T)>0$ such that 
    \[\liminf_{t \to \infty} \frac{\|\psi^{T,h}(t,\,\cdot\,)\|_{L^2_x}^2 - 1}{\|\varphi(t,\,\cdot\,)\|_{L^2_x}^2 - 1} > 1
    \]
    for all $0< h < h_0$.
\end{proposition}

Theorem~\ref{thm:torus} is then follows directly.

\begin{proof}[Proof of Theorem~\ref{thm:torus}]
    Note that for all $t>0$, $\|\varphi(t,\,\cdot\,)\|_{L^2_x}^2 -1 = \sum_{|k| \geq 1} \e^{- 8\pi^2 |k|^2 t} >0.$ Then by Proposition~\ref{prop:torus}, for suitably chosen $T,h$ and taking $t$ sufficiently large,
    \begin{align*}
        \frac{\|\psi^{T,h}(t,\,\cdot\,)\|_{L^2_x}^2 -\|\varphi(t,\,\cdot\,)\|_{L^2_x}^2}{ \|\varphi(t,\,\cdot\,)\|_{L^2_x}^2 - 1}  =  \frac{\|\psi^{T,h}(t,\,\cdot\,)\|_{L^2_x}^2 -1}{ \|\varphi(t,\,\cdot\,)\|_{L^2_x}^2 - 1} -1 >0,
    \end{align*}
    and thus
    \[\|\psi^{T,h}(t,\,\cdot\,)\|_{L^2_x} > \|\varphi(t,\,\cdot\,)\|_{L^2_x}.
    \]
    This is then directly~\eqref{eq:L2-torus-inversion}. Then~\eqref{eq:torus-not-preceq} follows from~\eqref{eq:L2-torus-inversion} and Proposition~\ref{prop:convex-dom} applied with $F(t) = t^2$. While the constructed vector field is not smooth in time, the arguments of Lemma~\ref{lem:approximate-by-mollification} show that a smooth vector field with the same properties can be obtained by mollification in time.
\end{proof}

The proof of Proposition~\ref{prop:torus} is as follows. For $t \leq T$, we have that $\psi^{T,h} = \varphi$. We then evaluate the (right) time derivative of \textit{low mode mass} of $\psi^{T,h}$ at time $T$, whereby low mode mass, we mean the squared $L^2$ on the lowest non-zero Fourier modes $|k| = 1$. Since we have an exact formula for $\psi^{T,h}(T,\,\cdot\,) = \varphi(T,\,\cdot\,)$, we can exactly compute this time derivative. $\hat u^{T,h}$ was chosen as effectively the simplest vector field for which this time derivative is larger than the time derivative of the low mode mass of $\varphi(T,\,\cdot\,)$. Then by Taylor's theorem, for sufficiently short times after $T$, we must have that the low mode mass of $\psi^{T,h}$ is larger than the low mode mass of $\varphi$. Choosing $h$ small enough to then turn off the advecting flow at this time, we have for times $[T+h,\infty)$ that $\psi^{T,h}, \varphi$ are just both freely diffusing under the heat equation. Since $\hat \psi^{T,h}(t,0) = \hat \varphi(t,0) = 1$ and the heat equation kills larger Fourier modes at a faster exponential rate, which of $\psi^{T,h}, \varphi$ have a larger $L^2$ norm at large times is determined by which has more low mode mass at $t = T+h$. Thus by construction, at large times $\psi^{T,h}$ has a larger $L^2$ norm than~$\varphi$. 

\begin{proof}[Proof of Proposition~\ref{prop:torus}]
    Note that 
    \[\hat \varphi(t,k) = \e^{- 4 \pi^2 |k|^2 t}\]
    and so 
    \[\frac{\dd}{\dd t} \sum_{|k|=1} |\hat \varphi|^2(T,k) = - 32 \pi^2 T \e^{- 8 \pi^2 T}.\]
    On the other hand, a direct computation in Fourier space gives that
    \[\frac{\dd}{\dd t}\sum_{|k| = 1} |\hat \psi^{T,h}|^2(T^+,k) = - 32\pi^2 T \e^{- 8 \pi^2 T} + 8\pi  \e^{-4\pi^2 T}\big(\e^{-8 \pi^2 T} -2 \e^{-16 \pi^2 T} + 2 \e^{-32 \pi^2 T} - \e^{-40 \pi^2 T}\big),\]
    where by $T^+$ we mean we compute the derivative from the right. Taking $T$ sufficiently large, we can ensure 
    \[\e^{-8 \pi^2 T} -2 \e^{-16 \pi^2 T} + 2 \e^{-32 \pi^2 T} - \e^{-40 \pi^2 T} \geq \frac{1}{2} \e^{-8 \pi^2 T},\]
    so that 
    \[\frac{\dd}{\dd t}\sum_{|k| = 1} |\hat \psi^{T,h}|^2(T^+,k) 
 \geq \frac{\dd}{\dd t} \sum_{|k|=1} |\hat \varphi|^2(T,k) +  \frac{1}{2} \e^{-8 \pi^2 T}.\]
    Taking then $h$ sufficiently small and using Taylor's theorem, we can ensure that
    \[\sum_{|k| = 1} |\hat \psi^{T,h}|^2(T+h,k) \geq  \sum_{|k|=1} |\hat \varphi|^2(T+h,k) + \frac{1}{4} \e^{-8 \pi^2 T} h.\]
    Then we note that for $t \geq T +h$, we since $u(t,x)=0$, 
    \begin{align*}\|\psi(t,\,\cdot\,)\|_{L^2_x}^2 -1 & \geq \sum_{|k| = 1} |\hat \psi^{T,h}|^2(t,k) 
    \\&= \e^{- 8 \pi^2 (t - (T+h))} \sum_{|k| = 1} |\hat \psi^{T,h}|^2(T+h,k)
    \\&\geq \e^{- 8 \pi^2 (t - (T+h))}  \sum_{|k|=1} |\hat \varphi|^2(T+h,k) + \frac{1}{4} \e^{- 8 \pi^2 (t - (T+h))}\e^{-8 \pi^2 T} h,
    \end{align*}
    while
    \begin{align*}\|\varphi(t,\,\cdot\,)\|_{L^2_x}^2 -1 &\leq \e^{- 8 \pi^2 (t - (T+h))} \sum_{|k| = 1} |\hat \varphi|^2 (T + h,k) + \e^{-16 \pi^2 (t - (T+h))} \sum_{|k| >1} |\hat \varphi|^2(T+h,k)
    \\&\leq  \e^{- 8 \pi^2 (t - (T+h))} \sum_{|k| = 1} |\hat \varphi|^2 (T + h,k) + \e^{-16 \pi^2 (t - (T+h))} C(T,h,d).
    \end{align*}
    Thus 
    \[\liminf_{t \to\infty}\frac{\|\psi(t,\,\cdot\,)\|_{L^2_x}^2 -1 }{\|\varphi(t,\,\cdot\,)\|_{L^2_x}^2 -1} \geq \liminf_{t \to\infty} \frac{\sum_{|k|=1} |\hat \varphi|^2(T+h,k) + \frac{1}{4}\e^{-8 \pi^2 T} h}{ \sum_{|k| = 1} |\hat \varphi|^2 (T + h,k) + \e^{-8\pi^2 (t - (T+h))} C(T,h,d)} > 1,\]
    as claimed.
\end{proof}

\begingroup
    {\small
    \bibliographystyle{alpha}
    \bibliography{keefer-references,references2}
    }
\endgroup
 
\end{document}